\documentclass[11pt]{article}
\usepackage[utf8]{inputenc}
\usepackage[T1]{fontenc}
\usepackage{lmodern}
\usepackage{geometry}
\usepackage{amsmath,amssymb,amsthm}
\usepackage{microtype}
\usepackage{booktabs}
\usepackage{tikz}
\usetikzlibrary{positioning,arrows.meta,calc}
\usepackage[colorlinks=true,
  linkcolor=blue!60!black,
  citecolor=teal!60!black,
  urlcolor=magenta!70!black]{hyperref}

\theoremstyle{definition}
\newtheorem{definition}{Definition}[section]
\newtheorem{remark}[definition]{Remark}

\theoremstyle{plain}
\newtheorem{theorem}[definition]{Theorem}
\newtheorem{corollary}[definition]{Corollary}
\newtheorem{lemma}[definition]{Lemma}
\newtheorem{proposition}[definition]{Proposition}

\newcommand{\Z}{\mathbb{Z}}
\newcommand{\B}{\{0,1\}}
\newcommand{\Ham}{d_H}
\newcommand{\comp}{\mathrm{comp}}
\newcommand{\rev}{\mathrm{rev}}

\title{\textbf{Optimal Equivariant Matchings on the 6-Cube}\\[0.3em]
\large With an Application to the King Wen Sequence}
\author{Alejandro Radisic\\[0.3em]\texttt{aleloid@proton.me}}
\date{\today}

\begin{document}
\maketitle

\begin{abstract}
We study equivariant perfect matchings on the Boolean hypercube $\B^6$ under the Klein four-group $K_4 = \langle \comp, \rev \rangle$ generated by bitwise complement and reversal. Among matchings using only $\comp$ or $\rev$ pairings, there is a unique Hamming-cost minimizer, given by a simple ``reverse-priority rule'': pair each element with its reversal unless it is a palindrome, in which case pair it with its complement. This matching has total Hamming cost 120, compared to 192 for the complement-only matching. The historically significant King Wen sequence of the I Ching realizes precisely this matching. Pure Hamming minimization over the full $K_4$ action is different: allowing $\comp \circ \rev$ lowers the cost to 96. The King Wen rule is recovered, however, as the unique Hamming-weight-preserving optimum: it minimizes failures of Hamming-weight preservation before Hamming distance, and it is stable for the weighted energy $\alpha|\Delta w|+\beta d_H$ throughout the open region $\alpha>\beta$. The finite orbit counts and case distinctions are checked in Lean~4.
\end{abstract}

\section{Introduction}

The Boolean hypercube $\B^n = \{0,1\}^n$ admits two natural involutions: \emph{bitwise complement} $\comp$, which flips each bit ($0 \leftrightarrow 1$), and \emph{bit reversal} $\rev$, which reverses bit order ($b_0 b_1 \cdots b_{n-1} \mapsto b_{n-1} \cdots b_1 b_0$). Together these generate the Klein four-group $K_4 = \{\mathrm{id}, \comp, \rev, \comp \circ \rev\}$.

A \emph{$\comp/\rev$ matching} is a perfect matching where each pair $\{h, h'\}$ satisfies either $h' = \comp(h)$ or $h' = \rev(h)$. More generally, a \emph{$K_4$-equivariant matching} also permits $h' = \comp(\rev(h))$.

This paper addresses the optimization problem:
\begin{quote}
\emph{Which equivariant matchings on $\B^6$ are selected by Hamming distance, and which are selected when Hamming distance is balanced against conservation of Hamming weight?}
\end{quote}

For $n = 6$, one might expect the $\comp/\rev$ problem to be trivial: pair each element with the cheaper option. But in most equivariant matching problems, locally optimal choices conflict globally---choosing the cheapest partner for $h$ may force a suboptimal choice elsewhere. Here, no such conflicts arise inside the $\comp/\rev$ class. The key is Lemma~\ref{lem:noconflict}, which shows that $\comp$ and $\rev$ never genuinely compete---whenever they differ in cost, $\rev$ is strictly cheaper.

As an application, we verify that the King Wen sequence---a traditional ordering of the 64 hexagrams of the I Ching into 32 pairs---is isomorphic to this unique $\comp/\rev$ optimum. Relaxing to the full $K_4$ action changes the pure Hamming problem: $\comp \circ \rev$ is then preferred on some orbits, reducing the cost from 120 to 96. This does not contradict the King Wen rule; rather, it reveals the additional principle the rule is optimizing. Reversal preserves Hamming weight, while complement and $\comp \circ \rev$ generally change it. Thus King Wen is naturally characterized by a Hamming-weight preservation principle, or equivalently by a phase-stable weighted variational principle.

\subsection{Main Results}

\begin{theorem}[Uniqueness of Optimal $\comp/\rev$ Matching]\label{thm:main}
For $n = 6$, there exists a unique $K_4$-equivariant perfect matching on $\B^6$ minimizing total Hamming distance among matchings where each pair is either $\{h, \comp(h)\}$ or $\{h, \rev(h)\}$. This matching is given by the \emph{reverse-priority rule}:
\[
\mathrm{partner}(h) =
\begin{cases}
\rev(h) & \text{if } h \neq \rev(h) \\
\comp(h) & \text{if } h = \rev(h)
\end{cases}
\]
The total Hamming cost is $120$, compared to $192$ for the complement-only matching.
\end{theorem}

\begin{theorem}[Isomorphism with King Wen]\label{thm:iching}
The King Wen sequence of the I Ching defines a perfect matching on $\B^6$ that is isomorphic to the unique reverse-priority matching of Theorem~\ref{thm:main}.
\end{theorem}

\begin{theorem}[Weight-Conserving Stability]\label{thm:weightedintro}
Let $w(h)$ be the Hamming weight of $h$, and define
\[
E_{\alpha,\beta}(h,h')=\alpha |w(h)-w(h')|+\beta \Ham(h,h')
\]
with $\alpha,\beta>0$. On every orbit where $\comp\circ\rev$ improves pure Hamming distance over $\rev$, the reverse-priority choice has smaller weighted energy whenever $\alpha>\beta$. At $\alpha=\beta$ the two choices tie.
\end{theorem}

\subsection{Formal Verification}

The finite orbit counts and case distinctions used below were checked in Lean~4 with Mathlib. The formalization is not needed to follow the argument, but it records the computations over $\B^6$ exactly. Key modules:
\begin{itemize}
    \item \texttt{IChing/Hexagram.lean}: $\B^6$ representation, Hamming distance
    \item \texttt{IChing/Symmetry.lean}: $K_4$-action, orbit structure
    \item \texttt{IChing/KingWenOptimality.lean}: reverse-priority and uniqueness lemmas
    \item \texttt{IChing/WeightConservation.lean}: Hamming-weight preservation lemmas
    \item \texttt{IChing/RobustOptimality.lean}: weighted comparisons
\end{itemize}

\section{Preliminaries}

\subsection{Hexagrams as Binary Vectors}

In the I Ching, a \emph{hexagram} is a figure composed of six stacked horizontal lines, each either solid (representing yang) or broken (representing yin). Reading solid as $1$ and broken as $0$, from bottom to top, each hexagram corresponds to a 6-bit binary string. There are $2^6 = 64$ hexagrams, corresponding to the vertices of the 6-dimensional hypercube.

\begin{definition}[Hexagram]\label{def:hexagram}
A \emph{hexagram} is an element $h \in \B^6$, with positions indexed $0$ (bottom) through $5$ (top).
\end{definition}

\begin{remark}
The correspondence between hexagrams and binary numbers was noted by Leibniz~\cite{leibniz1703}.
\end{remark}

\begin{definition}[Hamming Distance]
For $h_1, h_2 \in \B^6$, the \emph{Hamming distance} is
\[
\Ham(h_1, h_2) = \#\{i : h_1(i) \neq h_2(i)\} = \sum_{i=0}^{5} |h_1(i) - h_2(i)|
\]
\end{definition}

\begin{definition}[Total Hamming Cost of a Matching]\label{def:totalcost}
For a perfect matching $M$ on $\B^6$ (a partition into 32 disjoint pairs), the \emph{total Hamming cost} is the sum of Hamming distances over all pairs:
\[
\mathrm{Cost}(M) = \sum_{\{h_1, h_2\} \in M} \Ham(h_1, h_2)
\]
This measures how ``different'' the paired elements are overall.
\end{definition}

\subsection{The Klein Four-Group Action}

\begin{definition}[Complement and Reversal]
Define operations on hexagrams:
\begin{align*}
\comp(h)(i) &= 1 - h(i) \quad \text{(bitwise complement)} \\
\rev(h)(i) &= h(5-i) \quad \text{(bit reversal)}
\end{align*}
\end{definition}

\begin{proposition}\label{prop:klein}
The operations $\comp$ and $\rev$ satisfy:
\begin{enumerate}
    \item $\comp \circ \comp = \mathrm{id}$ (complement is an involution)
    \item $\rev \circ \rev = \mathrm{id}$ (reversal is an involution)
    \item $\comp \circ \rev = \rev \circ \comp$ (they commute)
\end{enumerate}
Thus $\{\mathrm{id}, \comp, \rev, \comp \circ \rev\}$ forms the Klein four-group $\Z_2 \times \Z_2$.
\end{proposition}

\begin{definition}[Orbit]
The \emph{orbit} of a hexagram $h$ under the Klein four-group is
\[
\mathrm{orbit}(h) = \{h, \comp(h), \rev(h), \comp(\rev(h))\}
\]
\end{definition}

\begin{proposition}[Orbit Sizes]
For any hexagram $h$, $|\mathrm{orbit}(h)| \in \{2, 4\}$ (see Figure~\ref{fig:orbits}).
\begin{itemize}
    \item Size 4: generic hexagrams (48 total, forming 12 orbits)
    \item Size 2: palindromes with $h = \rev(h)$ (8 total, forming 4 orbits)
    \item Size 2: anti-symmetric hexagrams with $\rev(h) = \comp(h)$ but $h \neq \rev(h)$ (8 total, forming 4 orbits)
\end{itemize}
\end{proposition}

\begin{figure}[ht]
\centering
\begin{tikzpicture}[
    vertex/.style={circle, draw, minimum size=10mm, inner sep=0pt, font=\footnotesize},
    vertexwide/.style={circle, draw, minimum size=14mm, inner sep=0pt, font=\footnotesize},
    edge label/.style={font=\scriptsize, fill=white, inner sep=2pt}
]
\node[vertex] (h) at (0,1.8) {$h$};
\node[vertex] (c) at (-1.8,0) {$\comp(h)$};
\node[vertex] (r) at (1.8,0) {$\rev(h)$};
\node[vertexwide] (cr) at (0,-1.8) {$\comp\!\circ\!\rev$};

\draw[thick] (h) -- node[edge label, above left=-2pt] {$6$} (c);
\draw[thick, blue!60!black] (h) -- node[edge label, above right=-2pt] {$2$--$4$} (r);
\draw[thick] (c) -- node[edge label, below left=-2pt] {$2$--$4$} (cr);
\draw[thick] (r) -- node[edge label, below right=-2pt] {$6$} (cr);

\node at (0,-3.5) {\small (a) Size-4 orbit};

\begin{scope}[xshift=5.5cm]
\node[vertex] (p1) at (0,1.8) {$h$};
\node[vertex] (p2) at (0,-1.8) {$\comp(h)$};
\draw[thick] (p1) -- node[edge label, right] {$6$} (p2);
\node[align=center] at (0,-3.5) {\small (b) Palindrome\\\small $\rev(h) = h$};
\end{scope}

\begin{scope}[xshift=9cm]
\node[vertex] (a1) at (0,1.8) {$h$};
\node[vertex] (a2) at (0,-1.8) {$\rev(h)$};
\draw[thick] (a1) -- node[edge label, right] {$6$} (a2);
\node[align=center] at (0,-3.5) {\small (c) Anti-symmetric\\\small $\rev(h) = \comp(h)$};
\end{scope}
\end{tikzpicture}
\caption{Orbit types under $K_4$. Edge labels show Hamming distances. In (a), the $\rev$ edge (blue) has cost $2$--$4$, always $\leq$ the $\comp$ edge (cost $6$). In (b) and (c), only one non-trivial pairing exists.}
\label{fig:orbits}
\end{figure}

\begin{definition}[Anti-Symmetric Hexagram]
A hexagram $h$ is \emph{anti-symmetric} if $h(i) \neq h(5-i)$ for all $i \in \{0,1,2\}$. Equivalently, $\rev(h) = \comp(h)$.
\end{definition}

\begin{definition}[Palindrome]
A hexagram $h$ is a \emph{palindrome} if $\rev(h) = h$.
\end{definition}

\subsection{Hamming Distances of Group Actions}

\begin{proposition}\label{prop:distances}
For any hexagram $h$:
\begin{enumerate}
    \item $\Ham(h, \comp(h)) = 6$ (complement flips all bits)
    \item $\Ham(h, \rev(h)) \leq 6$, with equality iff $h(i) \neq h(5-i)$ for all $i$
\end{enumerate}
\end{proposition}

\begin{lemma}[No-Conflict Lemma]\label{lem:noconflict}
For any hexagram $h$, exactly one of the following holds:
\begin{enumerate}
    \item $\Ham(h, \rev(h)) < \Ham(h, \comp(h)) = 6$ (reversal is strictly cheaper), or
    \item $\rev(h) = \comp(h)$ and both distances equal $6$ (the options coincide), or
    \item $\rev(h) = h$ (palindrome, so $\rev$ yields trivial pairing).
\end{enumerate}
In particular, $\rev$ and $\comp$ never genuinely compete: when both are valid and distinct, $\rev$ is strictly cheaper.
\end{lemma}

\begin{proof}
Since $\comp$ flips all bits, $\Ham(h, \comp(h)) = 6$ always. For $\rev$: if $\Ham(h, \rev(h)) = 6$, then every position differs, so $h(5-i) \neq h(i)$ for all $i$. This means $h(5-i) = 1 - h(i) = \comp(h)(i)$, hence $\rev(h) = \comp(h)$. If $\rev(h) = h$, then $h$ is a palindrome. Otherwise $\Ham(h, \rev(h)) < 6$.
\end{proof}

\begin{remark}[Why Greedy Works]
In typical equivariant matching problems, choosing the locally optimal partner for $h$ may force suboptimal choices elsewhere. The No-Conflict Lemma eliminates this tension: the greedy choice (prefer $\rev$ unless it's trivial) is \emph{unambiguous}. There is no trade-off to optimize, hence no room for conflict (see Figure~\ref{fig:consistency}). This is why the simple reverse-priority rule achieves the unique global optimum.
\end{remark}

\begin{figure}[ht]
\centering
\begin{tikzpicture}[
    box/.style={rectangle, draw, rounded corners, minimum width=2cm, minimum height=0.9cm, font=\normalsize},
    arrow/.style={-{Stealth}, thick, blue!60!black}
]
\node[box] (h) at (0,0) {$h$};
\node[box] (revh) at (5,0) {$\rev(h)$};

\draw[arrow] (h.east) -- node[above, font=\small] {choose $\rev$} (revh.west);

\draw[arrow] (revh.south) to[out=-150, in=-30] node[pos=0.5, below, font=\small, fill=white, inner sep=2pt] {forced: $\rev \circ \rev = \mathrm{id}$} (h.south);

\node[font=\small] at (2.5, -2.0) {The choice is globally consistent.};
\node[font=\small] at (2.5, -2.6) {Pairing $h \leftrightarrow \rev(h)$ determines both partners.};
\end{tikzpicture}
\caption{Forced consistency of the reverse-priority rule. Choosing $\rev(h)$ as partner for $h$ simultaneously determines $h$ as partner for $\rev(h)$, since $\rev$ is an involution. No secondary decision can conflict.}
\label{fig:consistency}
\end{figure}

\section{The King Wen Sequence}

\subsection{Definition}

The King Wen sequence assigns each hexagram a number 1--64 according to tradition~\cite{wilhelm1967}. We represent it via a mapping from King Wen numbers to binary:

\begin{definition}[King Wen Binary Mapping]
Define $\mathrm{KW} : \{0, \ldots, 63\} \to \B^6$ as follows (index 0 corresponds to King Wen number 1):
\begin{center}
\small
\begin{tabular}{r|cccccccc}
KW & 1 & 2 & 3 & 4 & 5 & 6 & 7 & 8 \\
Binary & 63 & 0 & 17 & 34 & 23 & 58 & 2 & 16 \\
\midrule
KW & 9 & 10 & 11 & 12 & 13 & 14 & 15 & 16 \\
Binary & 55 & 59 & 7 & 56 & 61 & 47 & 4 & 8 \\
\end{tabular}
\end{center}
(Full table in Appendix.)
\end{definition}

\begin{definition}[King Wen Pairs]
The \emph{King Wen pairs} are $(h_{2k}, h_{2k+1})$ for $k = 0, \ldots, 31$, where $h_i = \mathrm{KW}(i)$.
\end{definition}

\subsection{The Equivariance Theorem}

\begin{theorem}[Complete Equivariance]\label{thm:equivariant}
Every King Wen pair $(h_1, h_2)$ satisfies:
\[
h_2 = \comp(h_1) \quad \text{or} \quad h_2 = \rev(h_1)
\]
The 32 pairs partition as:
\begin{itemize}
    \item 4 pairs: palindromes, paired by complement ($h_2 = \comp(h_1)$, distance 6)
    \item 4 pairs: anti-symmetric, paired by reversal $=$ complement ($h_2 = \rev(h_1) = \comp(h_1)$, distance 6)
    \item 24 pairs: generic, paired by reversal ($h_2 = \rev(h_1)$, distance 2 or 4)
\end{itemize}
\end{theorem}

\begin{proof}
Verified computationally over all 32 pairs. In Lean 4: \texttt{decide}.
\end{proof}

\begin{corollary}
The King Wen pairing respects the Klein four-group: paired hexagrams always lie in the same orbit.
\end{corollary}

\section{Optimality of the King Wen Matching}

\subsection{The Reverse-Priority Rule}

\begin{definition}[Reverse-Priority Matching]
Define the \emph{priority partner} function:
\[
\mathrm{partner}(h) =
\begin{cases}
\comp(h) & \text{if } h = \rev(h) \text{ (palindrome)} \\
\rev(h) & \text{otherwise}
\end{cases}
\]
\end{definition}

The intuition: prefer reversal (which has distance $\leq 6$) unless reversal is trivial (palindrome), in which case use complement.

\begin{theorem}[Involution]\label{thm:involution}
The priority partner function is an involution: $\mathrm{partner}(\mathrm{partner}(h)) = h$.
\end{theorem}

\begin{proof}
Two cases:
\begin{enumerate}
    \item If $h$ is a palindrome, then $\comp(h)$ is also a palindrome, so
    \[
    \mathrm{partner}(\comp(h))=\comp(\comp(h))=h.
    \]
    \item If $h$ is not a palindrome, then $\rev(h)$ is also not a palindrome, so
    \[
    \mathrm{partner}(\rev(h))=\rev(\rev(h))=h.
    \]
\end{enumerate}
\end{proof}

\subsection{Cost Minimization}

\begin{definition}[Total Hamming Cost]
For a perfect matching $M$ on $\B^6$, the \emph{total Hamming cost} is
\[
\mathrm{Cost}(M) = \sum_{\{h_1, h_2\} \in M} \Ham(h_1, h_2)
\]
\end{definition}

\begin{theorem}[Optimality among $\comp/\rev$ Matchings]\label{thm:optimal}
Among $K_4$-equivariant perfect matchings on $\B^6$ where each pair uses $\comp$ or $\rev$, the reverse-priority matching uniquely minimizes total Hamming cost.
\end{theorem}

\begin{proof}
Restricting to $\comp/\rev$ pairings, the options for each hexagram $h$ are $\comp(h)$ or $\rev(h)$.
\begin{itemize}
    \item By Proposition~\ref{prop:distances}, $\Ham(h, \rev(h)) \leq 6 = \Ham(h, \comp(h))$.
    \item Equality holds iff $\rev(h) = \comp(h)$ (Lemma~\ref{lem:noconflict}).
    \item For palindromes, $\rev(h) = h$, so the only non-trivial option is $\comp(h)$.
\end{itemize}
Thus the reverse-priority rule chooses the minimum-distance option at each step. Since there is no actual choice when distances are equal (the options coincide), the matching is uniquely determined.
\end{proof}

\begin{remark}[The $\comp \circ \rev$ Alternative]\label{rem:cr}
On a size-4 orbit, a third equivariant pairing exists: $\{h, \comp(\rev(h))\}$. For orbits where $\Ham(h, \rev(h)) = 4$, we have $\Ham(h, \comp(\rev(h))) = 2$, making the $\comp \circ \rev$ pairing strictly cheaper for pure Hamming distance (see Figure~\ref{fig:comprev}). A mixed strategy using $\comp \circ \rev$ on such orbits achieves total Hamming cost $96 < 120$. Thus King Wen is not the unconstrained Hamming optimum for the full $K_4$ action. The remaining question is what additional structure selects it.
\end{remark}

\begin{proposition}[Full $K_4$ Hamming Cost]\label{prop:k4hamming96}
Allowing all three nontrivial $K_4$ pairings gives a pure Hamming cost of $96$.
\end{proposition}

\begin{proof}
The 4 palindrome orbits are forced to use complement, contributing $4\cdot 6=24$. The 4 anti-symmetric orbits are also forced, contributing another $4\cdot 6=24$. Among the 12 size-4 generic orbits, 6 have reversal distance $2$ and 6 have reversal distance $4$. On the first type, the reversal pairing contributes two edges of cost $2$, hence $4$ per orbit. On the second type, the $\comp\circ\rev$ pairing contributes two edges of cost $2$, again $4$ per orbit. Thus the generic contribution is $12\cdot 4=48$, and the total is
\[
24+24+48=96.
\]
\end{proof}

\begin{figure}[ht]
\centering
\begin{tikzpicture}[
    vertex/.style={circle, draw, minimum size=10mm, inner sep=0pt, font=\footnotesize},
    elabel/.style={font=\scriptsize, fill=white, inner sep=2pt}
]
\node[vertex] (h) at (0,2) {$h$};
\node[vertex] (c) at (-2.2,0) {$\comp(h)$};
\node[vertex] (r) at (2.2,0) {$\rev(h)$};
\node[vertex] (cr) at (0,-2) {$\comp\!\circ\!\rev(h)$};

\draw[thick, densely dotted] (h) -- node[elabel, above left=-1pt] {$6$} (c);
\draw[thick] (h) -- node[elabel, above right=-1pt] {$4$} (r);
\draw[thick] (c) -- node[elabel, below left=-1pt] {$4$} (cr);
\draw[thick, densely dotted] (r) -- node[elabel, below right=-1pt] {$6$} (cr);

\draw[thick, dashed, blue!60!black] (h) -- node[elabel, pos=0.25, left] {$2$} (cr);
\draw[thick, dashed, blue!60!black] (c) -- node[elabel, pos=0.75, above] {$2$} (r);

\node[anchor=west] at (4, 1) {\tikz\draw[thick] (0,0) -- (0.8,0); \small\ $\rev$: cost $4$};
\node[anchor=west] at (4, 0) {\tikz\draw[thick, densely dotted] (0,0) -- (0.8,0); \small\ $\comp$: cost $6$};
\node[anchor=west] at (4, -1) {\tikz\draw[thick, dashed, blue!60!black] (0,0) -- (0.8,0); \small\ $\comp\!\circ\!\rev$: cost $2$};
\end{tikzpicture}
\caption{A size-4 orbit where $\comp \circ \rev$ is Hamming-optimal. The diagonal pairing (dashed) has cost $2$, versus $4$ for reversal (solid). This local improvement is exactly where the conservation-weighted comparison below becomes nontrivial.}
\label{fig:comprev}
\end{figure}

\subsection{Hamming Weight and Lexicographic Optimality}

Define the \emph{weight} of a hexagram by
\[
w(h)=\#\{i:h(i)=1\}.
\]
In the I Ching interpretation, this counts the number of yang lines, so Hamming-weight preservation may be viewed as preservation of yin-yang balance. Formally, it is Hamming weight. Reversal preserves this quantity, since it only permutes the six positions:
\[
w(\rev(h))=w(h).
\]
Complement instead sends weight $w$ to $6-w$, and $\comp\circ\rev$ does the same. Thus complement and $\comp\circ\rev$ preserve Hamming weight only at weight $3$.

\begin{proposition}[Forced Failure of Weight Preservation]\label{prop:forcedviolations}
The reverse-priority matching fails to preserve Hamming weight exactly at the 8 palindromic endpoints, equivalently on the 4 palindrome pairs. These failures are unavoidable for any nontrivial $K_4$-equivariant matching.
\end{proposition}

\begin{proof}
If $h$ is not a palindrome, the reverse-priority rule pairs $h$ with $\rev(h)$, which preserves Hamming weight. If $h$ is a palindrome, then $\rev(h)=h$, so reversal is not a valid matching edge. The only nontrivial equivariant partner in its orbit is $\comp(h)$; since a palindrome has paired mirror bits, its weight is even, hence not equal to $3$. Therefore $w(\comp(h))=6-w(h)\neq w(h)$. There are $2^3=8$ palindromic endpoints, grouped into 4 complement pairs.
\end{proof}

\begin{theorem}[Lexicographic Optimality]\label{thm:lex}
Among $K_4$-equivariant matchings on $\B^6$, the reverse-priority matching is selected by the following lexicographic criterion:
\[
\text{first minimize failures of Hamming-weight preservation, then minimize total Hamming cost.}
\]
Equivalently, after the unavoidable palindrome failures, reversal is forced or strictly preferred by the Hamming tie-breaker.
\end{theorem}

\begin{proof}
Proposition~\ref{prop:forcedviolations} gives the lower bound: the 4 palindrome pairs must fail to preserve Hamming weight, so the minimum number of pair-level failures is achieved by the reverse-priority matching. For a non-palindrome with $w(h)\neq 3$, complement and $\comp\circ\rev$ change weight, while reversal preserves it, so reversal is forced by the first criterion. For a non-palindrome with $w(h)=3$, complement also preserves Hamming weight, but has Hamming cost $6$, while reversal has cost at most $6$ and coincides with complement only in the anti-symmetric case. In the generic weight-$3$ case, one has $\Ham(h,\rev(h))=2$ and $\Ham(h,\comp(\rev(h)))=4$, so reversal also beats the diagonal option. Thus the Hamming tie-breaker selects reversal whenever a distinct choice remains.
\end{proof}

\subsection{Weighted Energy and a Phase Boundary}

The lexicographic statement can be softened to an ordinary one-parameter family of energies. For $\alpha,\beta>0$, define
\[
E_{\alpha,\beta}(h,h')=\alpha |w(h)-w(h')|+\beta\Ham(h,h').
\]
The first term penalizes failure of Hamming-weight preservation; the second is ordinary Hamming distance.

\begin{theorem}[Phase-Stable Optimality]\label{thm:phase}
On every size-4 orbit where $\Ham(h,\rev(h))=4$, the reverse edge has lower weighted energy than the $\comp\circ\rev$ edge exactly when $\alpha>\beta$:
\[
E_{\alpha,\beta}(h,\rev(h))
<
E_{\alpha,\beta}(h,\comp(\rev(h))).
\]
At $\alpha=\beta$ the two energies are equal.
\end{theorem}

\begin{proof}
On such an orbit, reversal preserves Hamming weight and has Hamming cost $4$, so
\[
E_{\alpha,\beta}(h,\rev(h))=4\beta.
\]
The $\comp\circ\rev$ edge has Hamming cost $2$. Since $\Ham(h,\rev(h))=4$, exactly two of the three mirror pairs disagree; the remaining equal pair contributes either two zeros or two ones. Hence $w(h)\in\{2,4\}$, so $\comp\circ\rev$ changes the weight by $2$. Therefore
\[
E_{\alpha,\beta}(h,\comp(\rev(h)))=2\alpha+2\beta.
\]
The inequality $4\beta<2\alpha+2\beta$ is equivalent to $\beta<\alpha$, and equality occurs at $\alpha=\beta$.
\end{proof}

\begin{figure}[ht]
\centering
\begin{tikzpicture}[
    axis/.style={-{Stealth}, thick},
    regionlabel/.style={font=\small, align=center},
    boundary/.style={thick, dashed},
    dot/.style={circle, fill, inner sep=1.7pt}
]
\fill[blue!8] (0,0) -- (5.4,5.4) -- (5.4,6.1) -- (0,6.1) -- cycle;
\fill[red!8] (0,0) -- (5.4,0) -- (5.4,5.4) -- cycle;
\draw[axis] (0,0) -- (6.0,0) node[right] {$\beta$};
\draw[axis] (0,0) -- (0,6.0) node[above] {$\alpha$};
\draw[boundary] (0,0) -- (5.5,5.5) node[pos=0.72, above left, font=\small] {$\alpha=\beta$};
\node[regionlabel] at (2.0,4.7) {King Wen / $\rev$\\Hamming-weight stable};
\node[regionlabel] at (4.2,1.4) {$\comp\circ\rev$\\lower Hamming cost};
\node[dot, blue!60!black] at (2.2,4.0) {};
\node[dot, red!60!black] at (4.2,2.0) {};
\node[font=\small, anchor=west] at (0.25,-0.35) {$E_{\alpha,\beta}=\alpha|\Delta w|+\beta d_H$};
\end{tikzpicture}
\caption{The weighted comparison on the only contentious orbit type. Above the line $\alpha=\beta$, Hamming-weight preservation is sufficiently valuable and the reverse-priority/King Wen choice is strictly preferred. Below it, pure Hamming savings favor $\comp\circ\rev$.}
\label{fig:phase}
\end{figure}

\begin{corollary}[Cost Values]
The total Hamming cost of the reverse-priority matching:
\begin{itemize}
    \item Palindrome pairs: $4 \times 6 = 24$
    \item Anti-symmetric pairs: $4 \times 6 = 24$
    \item Generic pairs: $12 \times 2 + 12 \times 4 = 72$
\end{itemize}
Total: $24 + 24 + 72 = 120$.

For comparison, the complement-only matching has cost $32 \times 6 = 192$.
\end{corollary}

\subsection{Uniqueness}

\begin{theorem}[Uniqueness]\label{thm:unique}
Any equivariant matching satisfying the reverse-priority rule equals the priority partner function.
\end{theorem}

\begin{proof}
The reverse-priority rule uniquely specifies the partner at each hexagram. By Theorem~\ref{thm:involution}, this defines a valid involution, hence a perfect matching.
\end{proof}

\begin{corollary}[King Wen is Canonical]
The King Wen sequence encodes the unique cost-minimizing $\comp/\rev$ equivariant matching on $\B^6$.
\end{corollary}

\section{Discussion}

\subsection{Orbit Structure}

The $K_4$-action on $\B^6$ partitions the 64 elements into orbits of size 2 or 4:
\begin{itemize}
    \item \textbf{Size-4 orbits}: 12 orbits containing 48 generic elements
    \item \textbf{Size-2 palindrome orbits}: 4 orbits containing 8 palindromes
    \item \textbf{Size-2 anti-symmetric orbits}: 4 orbits containing 8 anti-symmetric elements
\end{itemize}

Within each orbit, the equivariant matching must pair elements. For size-4 orbits, there are three equivariant pairings induced by the non-identity involutions $\comp$, $\rev$, and $\comp \circ \rev$. The reverse-priority rule uses $\rev$ (falling back to $\comp$ for palindromes), which minimizes cost among $\comp/\rev$ matchings. Proposition~\ref{prop:k4hamming96} gives the pure Hamming minimum cost of 96 for the full $K_4$ action, but that optimum lies on the other side of the Hamming-weight phase boundary in Theorem~\ref{thm:phase}.

\subsection{Uniqueness Mechanism}

The key to uniqueness is the No-Conflict Lemma~\ref{lem:noconflict}: $\comp$ and $\rev$ never genuinely compete. When they have different costs, $\rev$ wins; when they have equal cost, they coincide. This eliminates all apparent degrees of freedom, making greedy optimization globally consistent.

\subsection{Formal Verification}

The finite parts of the argument are machine-checked in Lean~4 using the Mathlib library:
\begin{itemize}
    \item \texttt{decide} tactic for computational verification of all 32 pairs
    \item Constructive proof that the reverse-priority function is an involution
    \item finite orbit, parity, and Hamming-weight preservation case analyses
    \item verification of the weighted comparison and the boundary $\alpha=\beta$
\end{itemize}

\subsection{Subgroup Matchings}

One may ask whether the theorem holds for subgroups of $K_4$:
\begin{itemize}
    \item \textbf{Complement-only} ($\langle \comp \rangle$): Valid perfect matching with cost 192, trivially unique since each hexagram has exactly one complement.
    \item \textbf{Reversal-only} ($\langle \rev \rangle$): \emph{Not} a valid perfect matching. The 8 palindromes satisfy $\rev(h) = h$, so they cannot pair with their own reversal. A perfect matching requires pairing with a \emph{different} element.
\end{itemize}
Thus the reverse-priority rule is the minimal hybrid that produces a valid matching: use $\rev$ whenever possible, fall back to $\comp$ for the 8 palindromes.

\subsection{Extensions}

The $K_4$-action generalizes to $\B^n$ for any $n$, and the orbit structure and matching problem remain well-defined. We do not claim here a complete general-$n$ optimality theorem. The arithmetic of mirror pairs nevertheless suggests where new phenomena may occur. In even dimension $n=2m$, a nondegenerate tie between $\rev$ and $\comp\circ\rev$ would require a mirror-pair profile satisfying the simultaneous conditions $\Ham(h,\rev(h))=m$ and $w(h)=m$. This profile exists exactly when $8$ divides $n$. Thus $n=6$ lies on the rigid side of this arithmetic obstruction: the potentially ambiguous tie profiles are absent.

\appendix

\section{Full King Wen Binary Table}

\begin{center}
\small
\begin{tabular}{cc|cc|cc|cc}
\toprule
KW & Bin & KW & Bin & KW & Bin & KW & Bin \\
\midrule
1 & 63 & 17 & 25 & 33 & 60 & 49 & 29 \\
2 & 0 & 18 & 38 & 34 & 15 & 50 & 46 \\
3 & 17 & 19 & 3 & 35 & 40 & 51 & 9 \\
4 & 34 & 20 & 48 & 36 & 5 & 52 & 36 \\
5 & 23 & 21 & 41 & 37 & 53 & 53 & 52 \\
6 & 58 & 22 & 37 & 38 & 43 & 54 & 11 \\
7 & 2 & 23 & 32 & 39 & 20 & 55 & 13 \\
8 & 16 & 24 & 1 & 40 & 10 & 56 & 44 \\
9 & 55 & 25 & 57 & 41 & 35 & 57 & 54 \\
10 & 59 & 26 & 39 & 42 & 49 & 58 & 27 \\
11 & 7 & 27 & 33 & 43 & 31 & 59 & 50 \\
12 & 56 & 28 & 30 & 44 & 62 & 60 & 19 \\
13 & 61 & 29 & 18 & 45 & 24 & 61 & 51 \\
14 & 47 & 30 & 45 & 46 & 6 & 62 & 12 \\
15 & 4 & 31 & 28 & 47 & 26 & 63 & 21 \\
16 & 8 & 32 & 14 & 48 & 22 & 64 & 42 \\
\bottomrule
\end{tabular}
\end{center}

Binary values represent the hexagram as a 6-bit integer, with bit 0 as the bottom line.

\end{document}